\documentclass[12pt, reqno]{amsart}
\usepackage{amssymb,amsmath,amsopn,textcomp, thmtools}

\usepackage[backgroundcolor=white, bordercolor=blue, 
linecolor=blue]{todonotes}
\usepackage[normalem]{ulem}

\usepackage[a4paper,margin=1in,footskip=0.25in]{geometry}
\usepackage[ansinew]{inputenc}
\usepackage[english]{babel}
\usepackage{mathrsfs}
\usepackage{setspace}
\usepackage{stmaryrd}
\usepackage[titletoc]{appendix}
\usepackage[colorlinks=true, linkcolor=blue, citecolor=blue]{hyperref}

\usepackage{dsfont}
\usepackage{bbm}

\usepackage{graphicx}
\usepackage{color}
\usepackage{caption,rotating}

\newtheorem{theorem}{Theorem}[section]
\newtheorem{corollary}[theorem]{Corollary}
\newtheorem{proposition}[theorem]{Proposition}
\newtheorem{lemma}[theorem]{Lemma}
\newtheorem*{theorem*}{Theorem}

\theoremstyle{definition}
\newtheorem{definition}[theorem]{Definition}

\newtheorem{remark}[theorem]{Remark}

\renewenvironment{proof}[1][Proof]{\noindent\textit{#1.} }{\hfill 
	\rule{0.5em}{0.5em}}

\addto\extrasenglish{}
\addto\extrasenglish{} 

\numberwithin{equation}{section}

\usepackage{color}

\newcommand{\cI}{\mathcal{I}}

\newcommand{\cX}{\mathcal{X}}
\newcommand{\cN}{\mathcal{N}}

\newcommand{\cE}{\mathcal{E}}
\newcommand{\cF}{\mathcal{F}}

\newcommand{\cL}{\mathcal{L}}
\newcommand{\cV}{\mathcal{V}}
\newcommand{\bP}{\mathbb{P}}
\newcommand{\bE}{\mathbb E}
\newcommand{\bR}{\mathbb R}

\def\scE{{\mathscr E}}
\def\scF{{\mathscr F}}

\def\scB{{\mathscr B}}
\def\wt{\widetilde}

\def\rd{\mathrm{d}}

\newcommand{\N}{\mathbb{N}}
\newcommand{\bS}{\mathbb{S}}
\newcommand{\R}{\mathbb{R}}

\newcommand{\Rd}{{\mathbb{R}^d}}

\newcommand{\E}{\mathbb{E}}
\newcommand{\1}{\mathbbm 1}

\newcommand{\eps}{\varepsilon}
\def\nn{\nonumber} 
\DeclareMathOperator*{\dist}{dist}

\author{Jaehoon Kang }
\email{jaehoon.kang@kaist.ac.kr}
\address{Stochastic Analysis and Application Research Center, Korea Advanced Institute of Science and Technology, 291 Daehak-ro, Yuseong-gu, Daejeon, 34141, Republic of Korea}

\title[Heat kernel estimates for anisotropic symmetric jump processes ]
{Heat kernel estimates for\\ anisotropic symmetric jump processes}

\begin{document}

\thanks{
	This work was supported by the National Research Foundation of Korea(NRF) grant funded by the Korea government(MSIT)(No. 2019R1A5A1028324).}

	\begin{abstract}
We show two-sided bounds of heat kernel for anisotropic non-singular symmetric pure jump Markov process whose jump kernel $J(x,y)$ is comparable to $\frac{\1_{\cV}(x-y)}{|x-y|^{d+\alpha}}$, where $\cV$ is a union of symmetric cones, $0<\alpha<2$ and $x,y\in\R^d$. 

	\end{abstract}
	
	\maketitle

\setcounter{tocdepth}{1}		
\tableofcontents

\section{Introduction}
\subsection{Background}
Let $\cX$ be a Markov process in $\R^d$. The transition density function for $\cX$ is a measurable function $\wt p:(0,\infty)\times \R^d\times\R^d\to(0,\infty)$ satisfying 
$$\bP^x(\cX_t\in A)=\int_A \wt p(t,x,y)\rd y,\quad\text{for}\;\;A\in\scB,$$
where $\scB$ is a Borel set. 
Thus, transition density function shows distribution of Markov process. 
Transition density function is called heat kernel since it is the fundamental solution to the equation $\partial_t u=\cL u$, where $\cL$ is the infinitesimal generator of the Markov process.  For example, the transition density function $p^B(t,x,y)$ of standard Brownian motion in $\R^d$ is given by
$$p^B(t,x,y)=\frac{1}{(2\pi t)^{d/2}}\exp\left(-\frac{|x-y|^2}{2t}\right)$$
and it is well-known that $p^B(t,x,y)$ is the fundamental solution of the heat equation
$$\partial_tu-\frac12\Delta u=0,$$
where $\Delta$ is the Laplacian and $\frac12\Delta$ is the infinitesimal generator of standard Brownian motion. Except a few special cases, it is impossible to obtain explicit form of heat kernel. Thus, obtaining two-sided heat kernel bounds is important for studying Markov processes and their infinitesimal generator. 

\medskip

Heat kernel estimates for symmetric Markov processes in Euclidean space have been addressed in many research articles. In particular, during past several decades, symmetric Markov processes with discontinuous sample paths have been intensively studied. One of well-known example is a pure jump symmetric  Markov process $Z$ in $\R^d$, whose infinitesimal generator $\cL$ is given by
$$\cL u(x)=\mbox{p.v.}\int_{\R^d}(u(y)-u(x))j(x,y)\rd y,$$
where $j(x,y)=j(y,x)$ and for some $\alpha\in(0,2)$ and $c_1, c_2>0$ the relation $c_1|x-y|^{-d-\alpha}\le j(x,y)\le c_2|x-y|^{-d-\alpha}$ holds for $x\neq y$.
It is shown in \cite{CK03} that the heat kernel $p^Z(t,x,y)$ for $Z$ satisfies the following: there exist $C\ge1$ such that for all $ t> 0$ and $x,y\in\R^d$
$$C^{-1}t^{-d/\alpha}\left(1\wedge \frac{t}{|x-y|^{\alpha}}\right)^{1+d/\alpha}\le p^Z(t,x,y)\le Ct^{-d/\alpha}\left(1\wedge \frac{t}{|x-y|^{\alpha}}\right)^{1+d/\alpha}.$$
Research articles \cite{CK03, SV07, CKK08, CK10, CKK11, BGR14,  M16, GRT19, BKKL19} discuss symmetric discontinuous Markov processes with jump kernels that are comparable to isotropic functions. In these cases, the Markov processes can jump in any directions. Thus, for any $x,y\in\R^d$,  the Markov processes which start at $x$ can touch $y$ with one jump unless the jump kernels are truncated.

\medskip

On the other hand, we can see heat kernel estimates for symmetric Markov processes with singular jump measure in \cite{Xu13, KKK19, KK21}. The Markov process considered in \cite{Xu13, KKK19} can only jump in the directions parallel with coordinate axes. A typical example is $d$-dimensional singular $\alpha$-stable process with the L\'evy measure $\nu$ given by
\begin{align*}
\nu(\rd h)=\sum^{d}_{i=1}|h^i|^{-1-\alpha}\rd h^i\prod_{j\neq i}\delta_{\{0\}}(\rd h^j).
\end{align*}
Thus, the L\'evy measure is concentrated on the union of $1$-dimensional subsets of Euclidean space.
In this case, the process can move from $x$ to $y$ with one jump if and only if $y-x$ is parallel with an axis. Thus, the process which starts at $x$ needs more than one jump to touch $y$ if $y-x$ is not parallel with an axis. By the independence of coordinate processes and heat kernel estimates for $1$-dimensional symmetric $\alpha$-stable process, it is easy to see that  heat kernel $p_0(t,x,y)$ of $d$-dimensional singular $\alpha$-stable process satisfies the following: there exists $C\ge1$ such that for any $t>0$ and $x,y\in\R^d$
$$C^{-1}\prod^{d}_{i=1}t^{-1/\alpha}\left(1\wedge\frac{t}{|x^i-y^i|^\alpha}\right)^{1+1/\alpha}\le p_0(t,x,y)\le C \prod^{d}_{i=1}t^{-1/\alpha}\left(1\wedge\frac{t}{|x^i-y^i|^\alpha}\right)^{1+1/\alpha},$$
where $x=(x^1, x^2, \dots, x^d)$ and $y=(y^1, y^2, \dots, y^d)$.  
By \cite{Xu13, KKK19}, we see that   heat kernel estimates for singular Markov jump processes are robust.

\medskip
 
In this article, we study anisotropic non-singular symmetric pure jump processes. More precisely, we consider a symmetric function $J(x,y)$ which is comparable to $\frac{\1_{\cV}(x-y)}{|x-y|^{d+\alpha}}$ as the jump kernel for the jump process. Here, $\cV$ is a symmetric set given by the union of symmetric cones and $0<\alpha<2$.  Then, we see that the corresponding Markov process can jump from $x$ to a point $y\in\cV-\{x\}$ directly, i.e., the process can touch $y$ with one jump. However, this process starting at $x$ needs more than one jump to touch $y\in\R^d\setminus(\cV-\{x\})$. Thus, we can see that the behavior of Markov process with jump kernel $J(x,y)$ has properties of those of isotropic $\alpha$-stable process and singular $\alpha$-stable process, together. The goal of this article is to obtain two-sided sharp heat kernel bounds for the Markov process with jump kernel $J(x,y)$.

\medskip

Let us mention recent studies related to anisotropic Markov processes. 
Weak Harnack inequality and H\"older regularity of solutions to the equations with the nonlocal operators including anisotropic cases are shown in \cite{FK13, BKS19, CKWb19} using analytic approach.  Heat kernel estimates for L\'evy processes with the general L\'evy measures which cover anisotropic cases are discussed in \cite{KaSz15, KaSz19}. Heat kernel bounds for singular Markov process in \cite{Xu13, KKK19} are extended to more general processes in \cite{KK21}. The Markov process considered in \cite{KK21} can move continuously depending on direction.

\subsection{Setup and main result}
Let $d\ge2$ and $\bS(d-1)$ be the unit sphere in $\R^d$. For $\lambda\in \bS(d-1)$ and $\theta\in (0,\pi/2)$, we define cone in $\R^d$ with vertex $0\in\R^d$, axis $\lambda$ and  aperture $\theta>0$  by
\begin{align*}
\Gamma^+(\lambda,\theta)&:=\Big\{z\in\R^d\setminus\{0\} : \lambda\cdot \frac{z}{|z|}\in (\cos\theta,1]\Big\}.
\end{align*}
Using this, we also define
\begin{align*}
\Gamma^-(\lambda, \theta)&:=\Gamma^+(-\lambda,\theta),\quad
\Gamma(\lambda, \theta):=\Gamma^+(\lambda, \theta)\cup \Gamma^-(\lambda, \theta).
\end{align*}
We say $\Gamma(\lambda, \theta)$ is a symmetric cone. Note that for any $\lambda\in \bS(d-1)$ and $\theta\in (0,\pi/2)$, we have $\Gamma^+(\lambda,\theta)=\Gamma^-(-\lambda,\theta)$ and $0<\theta'\le \theta$ implies $\Gamma(\lambda, \theta')\subset \Gamma(\lambda, \theta)$. 
For $i\in \cI:=\{1,2,\dots, n\}$ with $n\in\N$, let $\lambda_i\in \bS(d-1)$, $\theta_i\in(0,\pi/2]$ and $\Gamma_i=\Gamma(\lambda_i, \theta_i)$. 
We define
\begin{align}\label{general_set}
\cV:=\cup_{i=1}^n \Gamma_i.
\end{align}
Then, $\cV$ is a symmetric set in the sense that $x\in \cV$ if and only if $-x\in \cV$. 
For $z\in\R^d$ and $A, B\subset \R^d$, we define
\begin{align*}
A_z&:=A+z:=\{x+z\,|\,x\in A\},\quad
A_B:=\cup_{z\in B}A_z.
\end{align*}
Note that for given $\cV$,  $y\in \cV_x$ if and only if $x\in \cV_y$ for all $x,y\in\R^d$. For $0<\alpha<2$ and $\cV$ given by \eqref{general_set}, define 
\begin{align*}
J^{\alpha}(x,y):=\frac{\1_{\cV}(x-y)}{|x-y|^{d+\alpha}}.
\end{align*}
Then, we see that $J^{\alpha}(x,y)=J^{\alpha}(y,x)$ for all $x,y\in\Rd$. Using this, we define the symmetric Dirichlet form on $L^2(\Rd)$ by 
\begin{align}\label{def:DF_Levy}\begin{split}
\scE(u,v)&:=\int_{\R^{d}}\int_{\R^{d}}(u(y)-u(x))(v(y)-v(x))J^{\alpha}(x, y)\,\rd y\rd x\\
\scF&:=\{u\in L^2(\R^{d}):\scE(u,u)<\infty\}.
\end{split}\end{align}
Then,  $(\scE, \scF)$ is a regular Dirichlet form by \cite[Theorem 2.4]{SU12}.

Let $\kappa\ge1$ and $J:\R^{d}\times \R^{d}\to (0,\infty)$ be a symmetric function satisfying
\begin{align}\label{J_comp}
{\kappa^{-1}}J^{\alpha}(x, y)\le J(x, y)\le {\kappa }J^{\alpha}(x, y).
\end{align}
Define the symmetric Dirichlet form $(\cE, \cF)$ by
\begin{align}\label{def:DF}\begin{split}
\cE(u,v)&:=\int_{\R^{d}}\int_{\R^{d}}(u(y)-u(x))(v(y)-v(x))J(x, y)\,\rd y\rd x\\
\cF&:=\{u\in L^2(\R^{d}):\cE(u,u)<\infty\}.
\end{split}\end{align}
Then, we see by \eqref{J_comp} that for any $u\in L^2(\Rd)$,
\begin{align}\label{E_comp}
\kappa^{-1}\cE(u,u)\le \scE(u,u)\le \kappa\cE(u,u)
\end{align}
and thus $\cF=\scF$. This shows that $(\cE, \cF)$ is also a regular Dirichlet form. Thus, there exists a Hunt process $X=X_{\cV}$ associated with $(\cE, \cF)$. The following is the main result of this article. For any $x\in\R^d$ and $A\subset \R^d$, $\dist(x, A):=\inf\{|x-w|:w\in A\}$ is the distance between $x$ and $A$.

\begin{theorem}\label{t:main}
Let $d\ge2$, $\alpha\in(0,2)$ and $\cV$ be a symmetric set given by \eqref{general_set}. Suppose $J$ satisfies \eqref{J_comp} and $(\cE, \cF)$ is the Dirichlet form given by \eqref{def:DF}. Then, there is a conservative Feller process $X=(X_t, \bP^x, x\in \R^{d}, t \ge 0)$ associated with $(\cE, \cF)$ that starts every point in $\bR^{d}$. Moreover,  $X$ has a continuous transition density function $p(t,x,y)$ on $(0,\infty)\times\bR^{d}\times\bR^{d}$, with the following estimates: there exists  $C\ge1$ such that for any $(t,x,y)\in (0,\infty)\times \R^{d} \times \R^{d}$,
\begin{align}\begin{split}\label{e:main}
&C^{-1}t^{-d/\alpha}\left(1\wedge\frac{t}{|x-y|^{\alpha}}\right)^{1+d/\alpha}\left(1\wedge\frac{t}{(\dist(y-x, \cV))^{\alpha}}\right)\\
&\le p(t,x,y)
\le C t^{-d/\alpha}\left(1\wedge\frac{t}{|x-y|^{\alpha}}\right)^{1+d/\alpha}\left(1\wedge\frac{t}{(\dist(y-x, \cV))^{\alpha}}\right).
\end{split}
\end{align}
\end{theorem}

\smallskip

The proof of \autoref{t:main} is given at the end of \autoref{s:lbe}. Note that $\dist(y-x, \cV)=\dist(x-y, \cV)$ since $\cV$ is symmetric. By \eqref{e:main}, we see that if $\dist(y-x,\cV)\le t^{1/\alpha}$, then $p(t,x,y)$ is comparable to the heat kernel of $d$-dimensional isotropic $\alpha$-stable process.

\medskip

To express \eqref{e:main} in a different way, let us define set $S(\Gamma, x,y)$ for a symmetric cone $\Gamma$ and $x,y\in\R^d$ as follows: 
\begin{definition}\label{d:setD}
 Let $\Gamma$ be a symmetric cone and $x,y\in \Rd$. We define
\begin{align*}
&S(\Gamma, x,y)\\
&:=\begin{cases}
\{x,y\}&\text{if}\;\;y-x\in\Gamma,\\
\{z\in \partial\Gamma_x\cap \partial\Gamma_y:|x-z|+|y-z|\le |x-u|+|y-u|, u\in\Gamma_x\cap \Gamma_y \}&\text{if}\;\;y-x\notin\Gamma.
\end{cases}
\end{align*}
\end{definition}

Note that if $y-x\notin\Gamma$, then the Hunt process $X_{\Gamma}$ can not jump from $x$ to $y$ directly. However, $X_{\Gamma}$ which starts at $x$ can touch $y$ in two jumps through $z\in S(\Gamma, x,y)$. 
Since $\Gamma$ is a symmetric cone, we see that $S(\Gamma, x,y)$ has always two points.

\begin{remark}\label{r:setD}
Let $\Gamma=\Gamma(\lambda, \theta)$ be a symmetric cone and  $x,y\in\R^d$. 
\begin{itemize}
\item[(i)] Let $z, \bar z\in S(\Gamma, x,y)$ with $z\neq \bar z$. Then, we see that 
\begin{align*}
|x-z|=|y-\bar z|,\quad |x-\bar z|=|y-z|
\end{align*}
holds. Thus, 
\begin{align*}
|x-z|\vee|y-z|&=|x-\bar z|\vee|y-\bar z|,\nn\\
|x-z|\wedge|y-z|&=|x-\bar z|\wedge|y-\bar z|,
\end{align*}
where $a\wedge b:=\min \{a, b\}$ and $a\vee b:=\max\{a, b\}$ for $a, b\in \R$. 
\item[(ii)] There exists $C_1=C_1(\theta)>0$ such that for $z\in S(\Gamma, x,y)$
\begin{align}\label{dist_comp}
C_1(|x-z|\vee |y-z|)\le |x-y|\le 2(|x-z|\vee |y-z|).
\end{align}
\item[(iii)]
There exists $C_2=C_2(\theta)\ge1$ such that for any $x,y\in\Rd$ with $y-x\notin\Gamma$ and $z\in S(\Gamma, x,y)$, we have
\begin{align*}
\mbox{dist}(y-x, \Gamma)\le |y-z|\wedge |x-z|\le C_2\, \mbox{dist}(y-x, \Gamma).
\end{align*}
If $\theta=\pi/4$, then $C_2=1$.
\item[(iv)] Let $\Gamma_1\subset \Gamma_2$ be two symmetric cones. Then, for $x,y\in\Rd$ and $z_i\in S(\Gamma_i, x, y)$
\begin{align*}
|x-z_1|\wedge|y-z_1|\ge |x-z_2|\wedge|y-z_2|.
\end{align*}
\end{itemize}
\end{remark}

Let $\cV$ be a symmetric set defined by \eqref{general_set}. By \autoref{r:setD}, we see that there exists $C_*\ge1$ depending only on $\theta_1, \dots, \theta_n$ such that for any $x,y\in\Rd$, we have
\begin{align}\label{C_*}
C_*|x-y|\ge \dist(y-x, \cV).
\end{align}

By \autoref{t:main} and \autoref{r:setD}, we have the following:
\begin{corollary}\label{c:main}
Under the same assumptions in \autoref{t:main} with $\cV=\Gamma$ for a symmetric cone $\Gamma$,  there exists $C\ge1$ such that for any $(t,x,y)\in(0,\infty)\times\R^d\times\R^d$ and $z\in S(\Gamma, x,y)$,
\begin{align*}\begin{split}
&C^{-1}t^{-d/\alpha}\left(1\wedge\frac{t}{(|x-z|\vee|y-z|)^{\alpha}}\right)^{1+d/\alpha}\left(1\wedge\frac{t}{(|x-z|\wedge|y-z|)^{\alpha}}\right)\\
&\le p(t,x,y)
\le C t^{-d/\alpha}\left(1\wedge\frac{t}{(|x-z|\vee|y-z|)^{\alpha}}\right)^{1+d/\alpha}\left(1\wedge\frac{t}{(|x-z|\wedge|y-z|)^{\alpha}}\right).
\end{split}\end{align*}
\end{corollary}
Let $z\in S(\Gamma, x,y)$ be a point such that $|x-z|\ge |y-z|$. Then, by \autoref{c:main}, we see that the product of two term $(1\wedge\frac{t}{|x-z|^{\alpha}})^{1+d/\alpha}$ and $(1\wedge\frac{t}{|y-z|^{\alpha}})$ appears for off-diagonal estimates. The first term is about the jump from $x$ to $z$ and the second term is about the jump from $z$ to $y$. We can also see this product expression in the heat kernel estimates for singular Markov processes. This is due to the fact that the anisotropic jump process which starts at $x$ needs more than one jump to touch $y$.  

\medskip

On the other hand, we can check that the heat kernel bounds in \autoref{c:main} are distinguished from that of singular processes since the two terms $(1\wedge\frac{t}{|x-z|^{\alpha}})$ and $(1\wedge\frac{t}{|y-z|^{\alpha}})$ have different orders even though the jump intensity is $d+\alpha$ for all jumps.  Recall that heat kernel for $d$-dimensional singular $\alpha$-stable process is comparable to 
$$\prod^{d}_{i=1}t^{-1/\alpha}\left(1\wedge\frac{t}{|x^i-y^i|^\alpha}\right)^{1+1/\alpha}$$ 
for $x=(x^1, \dots, x^d)$ and $y=(y^1, \dots, y^d)$. Thus, each term for the off-diagonal estimates has the same order. 

\medskip

Define the Green function of $X$ by
\begin{align*}
G(x,y):=\int^{\infty}_{0}p(t,x,y)\rd t.
\end{align*}
By \autoref{t:main}, we obtain the following Green function estimates:
\begin{corollary}\label{c:GFE}
Under the same assumptions in \autoref{t:main}, there exists $C\ge1$ such that for any $x,y\in\R^d$
\begin{align*}
\frac{C^{-1}}{|x-y|^{d-\alpha}}\le G(x,y)\le \frac{C}{|x-y|^{d-\alpha}}.
\end{align*}
\end{corollary}
The proof of \autoref{c:GFE} is given at the end of \autoref{s:lbe}. By \autoref{c:GFE}, we see that the Green function of $X$ has the same estimates with that of isotropic $\alpha$-stable processes in $\R^d$. Thus, we see that there exist Markov processes such that the  corresponding Green functions are comparable but the corresponding heat kernels are not. 

\medskip

This article is organized as follows: 
In \autoref{s:upper}, we obtain near-diagonal upper bound for the heat kernel by using L\'evy process. Moreover, we prove two important auxiliary results. The first one is rough upper bounds for the heat kernel and the second one is a survival estimate. See \autoref{t:uhk}. In \autoref{s:sub}, we show sharp upper bounds for the heat kernel using the rough upper bounds and L\'evy system. 
In \autoref{s:lbe}, we give the lower bound estimates and the proof of \autoref{t:main} and \autoref{c:GFE}.

\medskip

\subsection*{Notation} 
For two non-negative functions $f$ and $g$, the notation $f\asymp g$ means that there are positive constants $c_1$ and $c_2$ such that $c_1g(x)\leq f (x)\leq c_2 g(x)$ in the common domain of definition for $f$ and $g$. Recall that  $a, b\in \R$, we use $a\wedge b$ for $\min \{a, b\}$ and $a\vee b$ for $\max\{a, b\}$. For any $x\in\R^d$ and $r>0$, $B(x,r):=\{y\in\R^d: |y-x|<r\}$.

\section{Upper bound estimates and auxiliary results}\label{s:upper}

We first show near-diagonal upper bounds for the heat kernel. Recall that $(\scE, \scF)$ is the Dirichlet form given in \eqref{def:DF_Levy}.
\begin{lemma}
There exists a constant $c>0$ such that for any $f\in \scF\cap L^1(\R^{d})$
\begin{align}\label{e:Nash_Levy}
\|f\|^{2+2(\alpha/d)}_{2}\le c \,\scE(f,f)\|f\|_1^{2(\alpha/d)}.
\end{align}
\end{lemma}
\begin{proof}
Let $Y$ be the symmetric L\'evy process associated with $(\scE, \scF)$. Then, the characteristic function $\phi$ of $Y$ is given by
\begin{align*}
\phi(\xi):=\int_{\Rd}(1-\cos(\xi\cdot z))\frac{\1_{\cV}(z)}{|z|^{d+\alpha}}\rd z=\int_{\cV}(1-\cos(\xi\cdot z))\frac{1}{|z|^{d+\alpha}}\rd z.
\end{align*}
Since $(1-\cos u)\ge (1-\cos1)u^2$ for $|u|\le1$, we see that
\begin{align*}
\phi(\xi)\ge (1-\cos1)\int_{\cV\cap \{|z|<1/|\xi|\}}\frac{|\xi\cdot z|^2}{|z|^{d+\alpha}}\rd z.
\end{align*}
Recall that $\Gamma_1=\Gamma(\lambda_1, \theta_1)$ is a symmetric cone satisfying $\Gamma_1\subset\cV$. Let $\Gamma_1'=\Gamma(\lambda_1, \theta_1')$ be a symmetric cone such that $\theta_1'\le\theta_1\wedge \pi/8$. Then, $\Gamma_1'\subset \Gamma_1$. 
Define
\begin{align*}
\Gamma_*:=
\begin{cases}
\Gamma_1',&\quad\text{if}\;\;\xi\in\Gamma(\lambda_1, \pi/4);\\
\Gamma_1'\cap \Gamma(\lambda_0, \theta_1'/4),&\quad\text{otherwise},
\end{cases}
\end{align*} 
where $\lambda_0\in\partial\Gamma_1'\cap\bS(d-1)$ such that $\xi\cdot\lambda_0\ge \xi\cdot\lambda'$ for all $\lambda'\in\partial\Gamma_1'\cap\bS(d-1)$.
Then, we can see that $|\xi\cdot z|\ge c_0|\xi||z|$ holds for all $z\in\Gamma_*$, where $c_0=\cos(3\pi/8)\wedge\cos(\pi/2-3\theta_1'/4)>0$. Thus,
\begin{align*}
\phi(\xi)&\ge (1-\cos1)\int_{\Gamma_*\cap \{|z|<1/|\xi|\}}\frac{|\xi\cdot z|^2}{|z|^{d+\alpha}}\rd z\\
&\ge c_0^2(1-\cos1)|\xi|^2\int_{\Gamma\cap \{|z|<1/|\xi|\}}\frac{|z|^2}{|z|^{d+\alpha}}\rd z=c_1|\xi|^{\alpha}.
\end{align*}
Note that $c_1>0$ is independent of $\xi$.
Thus, we have for any $t>0$
\begin{align}\label{ch_ftn}
\int_{\Rd}e^{-t \phi(\xi)}\rd \xi\le \int_{\Rd}e^{-c_1t |\xi|^{\alpha}}\rd \xi=t^{-d/\alpha}\int_{\Rd}e^{-c_1|z|^{\alpha}}\rd z=c_2t^{-d/\alpha}.
\end{align}
For $t>0$ and $x,y\in\Rd$, let $q(t, x, y):=q(t,y-x)$ be the heat kernel for $Y$, where 
\begin{align*}
q(t,x):=\frac{1}{(2\pi)^d}\int_{\Rd}e^{-t \phi(\xi)}e^{-i\langle x, \xi\rangle}\rd \xi.
\end{align*}
Then, by \eqref{ch_ftn}, we see that for any $t>0$ and $x\in\R^d$
\begin{align*}
q(t,x)\le \frac{1}{(2\pi)^d}\int_{\Rd}e^{-t \phi(\xi)}\rd \xi\le \frac{c_2}{(2\pi)^d}t^{-d/\alpha}.
\end{align*}
Thus, by \cite[Theorem 2.1]{CKS87}, we obtain \eqref{e:Nash_Levy}.
\end{proof}

\medskip

\begin{proposition}\label{p:uhkd}
(i) There exists a constant $c>0$ such that for any $f\in \cF\cap L^1(\R^{d})$
\begin{align}\label{e:Nash}
\|f\|^{2+2(\alpha/d)}_{2}\le c \,\cE(f,f)\|f\|_1^{2(\alpha/d)}.
\end{align}
(ii) There is a properly exceptional set $\cN$ of $X$, a positive symmetric kernel $p(t,x,y)$ defined on $(0,\infty)\times (\bR^{d}\setminus\cN)\times(\bR^{d}\setminus\cN)$, and a constant $C>0$ such that $\bE^x[f(X_t)]= \int_{\R^{d}}p(t,x,y)f(y)\rd y$, and
\begin{align}\label{e:uhkd}
p(t,x,y)\le C t^{-d/\alpha}
\end{align}
for every $x,y\in\bR^{d}\setminus\cN$ and for every $t>0.$
\end{proposition}
\begin{proof}
(i) By \eqref{e:Nash_Levy} and \eqref{E_comp}, we obtain \eqref{e:Nash}.

(ii) Using \eqref{e:Nash}, \cite[Theorem 2.1]{CKS87} and \cite[Theorem 3.1]{BBCK09}, the result follows.

\end{proof}

\medskip

\begin{remark}\label{r:Holder}
In \cite{FK13},  the H\"older continuity for solutions to parabolic equations was discussed. It is easy to see that the jump kernel $J(x,y)$ satisfies the conditions $\mathrm{K}_1$, $\mathrm{K}_2$ and $\mathrm{K}_3$ in \cite{FK13}. Since $p(t,x,y)$ is parabolic function, by \cite[Theorem 1.2]{FK13},  we may assume that $p(t,x,y)$ is H\"older continuous and $\cN=\emptyset$, where $\cN$ is the properly exceptional set in \autoref{p:uhkd}(ii). For more general results on H\"older estimates for nonlocal operators, see \cite{CKWb19}. 
\end{remark}

\medskip

Now we will give rough upper bound for the heat kernel. For ${\delta}>0$, define
$$J_{\delta}(x,y):=J(x,y)\1_{\{|x-y|\le {\delta}\}}$$
and for $u,v\in\cF$,
\begin{align*}
\cE_\delta(u,v)&:=\int_{\R^{d}}\int_{\R^{d}}(u(y)-u(x))(v(y)-v(x))J_\delta(x, y)\rd y \rd x.
\end{align*}
Let $X^{\delta}$ be the Markov process associated with $(\cE_\delta, \cF)$ and $p^{\delta}(t,x,y)$ be the heat kernel for $X^{\delta}$. 
\begin{lemma}\label{l:J_outball}
There exists $C>0$ such that for any $x\in\R^{d}$,
\begin{align*}
C^{-1}\delta^{-\alpha}\le \int_{\R^{d}}\big(J(x,y)-J_{\delta}(x,y)\big)\rd y\le C\delta^{-\alpha}.
\end{align*}
\end{lemma}
\begin{proof}
Using \eqref{J_comp},
\begin{align*}
&\int_{\R^{d}}\big(J(x,y)-J_{\delta}(x,y)\big) \rd y=\int_{|y-x|>\delta}J(x,y) \rd y\le \kappa\int_{|y-x|>\delta}\frac{1}{|x-y|^{d+\alpha}}\rd y=c\delta^{-\alpha}.
\end{align*}
On the other hand,
\begin{align*}
&\int_{\R^{d}}\big(J(x,y)-J_{\delta}(x,y)\big) \rd y=\int_{|y-x|>\delta}J(x,y) \rd y\ge \kappa^{-1}\int_{|y-x|>\delta}\frac{\1_{\Gamma_1}(x-y)}{|x-y|^{d+\alpha}}\rd y=c\delta^{-\alpha}.
\end{align*}

\end{proof}

\medskip

Using \autoref{l:J_outball},
\begin{align*}
\cE(u,u)-\cE_\delta(u,u)&=\int_{\R^{d}}\int_{\R^{d}}(u(x)-u(y))^2J(x,y)\1_{\{|x-y|>\delta\}}\,\rd y\rd x\\
&\le 4\int_{\R^{d}}u(x)^2\,\rd x \sup_x\int_{|y-x|>\delta}J(x,y)\,\rd y\\
&\le c\, \delta^{-\alpha}\|u\|_2^2.
\end{align*}
Thus, we obtain by \eqref{e:Nash} and \cite[Theorem 3.25]{CKS87} that
\begin{align}\label{p_trunc1}
p^{\delta}(t,x,y)\le c t^{-d/\alpha}e^{c_1 t \delta^{-\alpha}-E_\delta(2t,x,y)},
\end{align}
where
\begin{align*}
\Upsilon_{\delta}(f)(x)&:=\int_{\bR^{d}} (e^{f(x)-f(y)}-1)^2J_{\delta}(x,y)\rd y\\
\Lambda_{\delta}(f)^2&:=\|\Upsilon_{\delta}(f)\|_{\infty}\vee\|\Upsilon_{\delta}(-f)\|_{\infty},\\
E_{\delta}(t,x,y)&:=\sup\Big\{|f(x)-f(y)|-t\Lambda_{\delta}(f)^2: f\in \text{Lip}_c(\bR^d)\; \text{with}\;\Lambda_{\delta}(f)<\infty\Big \}.
\end{align*}

\begin{theorem}\label{t:uhk}
(i) There exists a constant $C>0$ such that for any $t>0$ and $x,y\in\R^d$
\begin{align}\label{e:uhk}
p(t,x,y)\le C t^{-d/\alpha}\left(1\wedge \frac{t}{|x-y|^{\alpha}}\right)^{1+d/\alpha}.
\end{align}
(ii) There exists a constant $c>0$ such that for any $t,r>0$ and $x\in\R^d$
\begin{align}\label{exit_est}
\bP^x(\tau_{B(x,r)}\le t)\le ctr^{-\alpha}.
\end{align}
\end{theorem}
\begin{proof}
(i) Fix $x,y\in\R^d$ and let $R:=|x-y|$. Let ${\delta}=\frac{R \alpha}{3(d+\alpha)}$. If $2t\ge \delta^{\alpha}$, then
\eqref{e:uhk} follows from \eqref{e:uhkd}. Thus, for the rest of the proof, we assume that $2t<\delta^{\alpha}$.
Let $\eta=\tfrac{1}{3\delta}\log({{\delta}^{\alpha}}/{t})$ and define
$$\psi(\xi)=\eta(R-|\xi-x|)\vee 0,\quad \text{for}\;\;\xi\in\R^{d}.$$
Then, by \eqref{J_comp}, 
\begin{align*}
\Upsilon_{\delta}({\psi})(\xi)&= \int_{\R^d}\Big(e^{\psi(\xi)-\psi(\zeta)}-1\Big)^2J_\delta(\xi, \zeta)\rd \zeta\nn\\
&\le\kappa\int_{|\xi-\zeta|\le\delta}\Big(e^{\psi(\xi)-\psi(\zeta)}-1\Big)^2J^{\alpha}(\xi, \zeta)\rd \zeta\nn\\
&=:\kappa I.
\end{align*}
For $\xi, \zeta\in\R^{d}$, we have $|\psi(\xi)-\psi(\zeta)|\le \eta|\xi-\zeta|$.
Using this and $(e^{s}-1)^2\le s^2e^{2|s|}$ for all $s\in \R$, we see that 
\begin{align*}
I=\int_{|\xi-\zeta|\le {\delta}}\Big(e^{\psi(\xi)-\psi(\zeta)}-1\Big)^2J^{\alpha}(\xi, \zeta)\rd \zeta
\le\int_{|s|\le {\delta}}\frac{\eta^2|s|^2e^{2\eta |s|}}{|s|^{d+\alpha}}\rd s
\le c\eta^2e^{2\eta \delta}\delta^{2-\alpha}
\le c e^{3\eta {\delta}}{\delta}^{-\alpha}.
\end{align*}
Thus, 
$\Upsilon_{\delta}(\psi)(\xi)\le \kappa ce^{3\eta \delta}{\delta}^{-\alpha}.$
Applying the same argument to $-\psi$, we also obtain $\Upsilon_{\delta}(-\psi)(\xi)\le \kappa ce^{3\eta \delta}{\delta}^{-\alpha}.$ 
Thus
\begin{align*}
-E_{\delta}(2t,x,y)&\le -\eta R+ \frac{\kappa ct}{\delta^{\alpha}}e^{3\eta \delta}=-\eta R+\kappa c.
\end{align*}
Using this, \eqref{p_trunc1} and $\delta^{\alpha}>2t$, we obtain
\begin{align}\label{trunc_upper}
p^{{\delta}}(t,x,y)&\le C t^{-d/\alpha}\exp(c t {\delta}^{-\alpha}-E_{\delta}(2t, x,y))\nn\\
&\le C t^{-d/\alpha}\exp\Big((1/2+\kappa)c-\eta R\Big)
= C' t^{-d/\alpha}\Big(\frac{t}{\delta^{\alpha}}\Big)^{(d+\alpha)/\alpha}
= C' \frac{t}{\delta^{d+\alpha}}\,.
\end{align}

From Meyer's construction and \autoref{l:J_outball},
\begin{align*}
p(t,x,y)\le p^{\delta}(t,x,y)+t\sup_{\xi\in\R^d}\|J_\delta(\xi,\cdot)\|_{\infty}\le c\frac{t}{\delta^{d+\alpha}}=\frac{c't}{|x-y|^{d+\alpha}}.
\end{align*}

\medskip

(ii) Using \eqref{e:uhk}, we obtain \eqref{exit_est}. The proof is standard and thus we omit it.

\end{proof}

\medskip

Let $t>0$ and $x, y\in \R^d$. Suppose $\dist(y-x, \cV)\le t^{1/\alpha}$. Then, the upper bound for the heat kernel in \eqref{e:main} follows from \autoref{t:uhk}(i) since 
\begin{align*}
1\wedge\frac{t}{(\dist(y-x, \cV))^{\alpha}}=1.
\end{align*}
Thus, for the sharp upper bound, it suffices to consider that $t>0$, $x,y\in\R^d$ with $\dist(y-x, \cV)> t^{1/\alpha}$. This happens only if $y-x\notin \cV$. We will prove the sharp upper bound for the case $x,y\in\Rd$ with $y-x\notin \cV$ in \autoref{s:sub}.

\section{Sharp upper bounds}\label{s:sub}

We first introduce L\'evy system for stochastic process with jumps. See, for example, \cite[Appendix A]{CK08}.
\begin{lemma}
For any $x\in \R^{d}$, stopping time $T$ (with respect to the filtration of $X$), and non-negative measurable function $f$ on $\R_+ \times \R^{d}\times \R^{d}$ with $f(s, y, y)=0$ for all $y\in\R^{d}$ and $s\ge 0$, we have 
	\begin{equation}\label{eq:LS}
	\E^x \left[\sum_{s\le T} f(s,X_{s-}, X_s) \right] = \E^x \left[ \int_0^T \int_{\R^d} 
	f(s,X_s, y) 
	J(X_s, y) \rd y  \rd s \right].
	\end{equation}
\end{lemma}

\medskip

The aim of this section is to prove the following: 
\begin{theorem}\label{t:sub}
There exists a constant $C>0$ such that for $t > 0$,  $x_0, y_0\in \R^d$ , 
\begin{align}\label{e:sub}
p(t,x_0,y_0)\le C t^{-d/\alpha}\left(1\wedge\frac{t}{|x_0-y_0|^{\alpha}}\right)^{1+d/\alpha}\left(1\wedge\frac{t}{(\dist(y_0-x_0,\cV))^{\alpha}}\right).
\end{align}
\end{theorem}

\medskip

Our approach is similar to that of \cite{KKK19, KK21} where the idea from \cite{BGK09} is used. 
For any two non-negative measurable functions $f, g$ on $\Rd$, set $$\langle f, g \rangle=\int _{\Rd}f(x)g(x)\rd x.$$ 

	\begin{lemma}\cite[Lemma 2.1]{BGK09}\label{lem:LBGK}
		Let $U$ and $V$ be two disjoint non-empty open subsets of $\Rd$ and $f, g$ be non-negative Borel functions on $\Rd$. 
		Let $\tau=\tau_U$ and $\tau^{'}=\tau_V$ be the first exit times from $U$ and $V$, respectively. Then, for all $a, b, t>0$ such that $a+b=t$, we have
		\begin{align}\label{eq:LBGK}
		\langle P_tf, g \rangle  \le \big\langle\E^{\cdot}\big[\1_{\{\tau\le a\}}P_{t-\tau}f(X_{\tau})\big], g\big\rangle+ \big\langle \E^{\cdot}\big[\1_{\{\tau^{'}\le b\}}P_{t-\tau}g(X_{\tau^{'}})\big], f\big\rangle\,.
		\end{align}
	\end{lemma}

The following proposition will play a key role in obtaining sharp upper bounds. Similar results for singular jump processes and direction-depending Markov processes are shown in \cite[Proposition 3.3]{KKK19} and \cite[Proposition 4.4, Proposition 4.9]{KK21}, respectively. Recall that $C_*\ge1$ is the constant in \eqref{C_*}.

\begin{proposition}\label{p:sub_key}
Let $t > 0$ and  $x_0, y_0\in \R^d$. Set $\rho=t^{1/\alpha}$ and $R_1=\dist(y_0-x_0,\cV)$. Let $f$ be a non-negative Borel function on $\R^d$ supported in $B(y_0, {\rho}/{(8C_*)})$. Define an exit time $\tau$ by $\tau=\tau_{B(x_0, {R_1}/{(8C_*)})}$. Then there exists $C>0$ independent of  $x_0, y_0$ and $t$ such that for every $x\in B(x_0, {\rho}/{(8C_*)})$,
	\begin{align}\label{e:sub_key}
	\begin{split}
	&\E^{x}\left[\1_{\{\tau\le t/2\}}P_{t-\tau}f(X_{\tau})\right]\\
	&\le \,C t^{-d/\alpha} \|f\|_1  \left(1\wedge\frac{t}{|x_0-y_0|^{\alpha}}\right)^{1+d/\alpha}\left(1\wedge\frac{t}{(\dist(y_0-x_0,\cV))^\alpha}\right).
	\end{split}
	\end{align}
\end{proposition}

\begin{proof}
Let $Q\in\partial \cV$ be a point such that $\dist(y_0-x_0, \cV)=|y_0-x_0-Q|=:R_1$.  If $R_1\le \rho=t^{1/\alpha}$, \eqref{e:sub_key} follows from \eqref{e:uhk}. Thus, for the rest of proof we assume that $y_0-x_0\notin\cV$ and $R_1>\rho$. Let $R_2:=|x_0-y_0|$. Then,  $\rho<R_1\le C_*R_2$.  Let $B:=B(x_0, R_1/(8C_*))$. Note that if the Markov process $X$ starting from $x\in B$ exits the ball $B$, $X_{\tau_B}$ should be in $\cV_B\setminus B$.  Then, by the definition of $Q$, we see that for $y\in B(y_0, \rho/(8C_*))$, 
\begin{align*}
\inf_{w\in\cV_B\setminus B}|y-w|\ge\inf_{w\in\cV_B}|y-w|\ge |y_0-Q|-|y_0-y|-\frac{R_1}{8C_*}\ge R_1-\frac{\rho}{8C_*}-\frac{R_1}{8C_*}\ge \frac{R_1}{2}.
\end{align*}
This implies that for $y\in B(y_0, \rho/(8C_*))$ and $w\in \cV_B\setminus B$,
\begin{align}\label{comp_yy0}
\frac12|y-w|\le |y-w|-|y_0-y|\le |y_0-w|\le |y_0-y|+|y-w|\le 2|y-w|.
\end{align}

By \eqref{e:uhk} and \eqref{comp_yy0},  for $\tau\le t/2$ and $w\in\cV_B\setminus B$
\begin{align}
\begin{split}\label{sg_bd}
P_{t-\tau}f(w)&=\int_{B(y_0, \rho/(8C_*))}p(t-\tau, w, y)f(y) \rd y\\
&\le C\int_{B(y_0, \rho/(8C_*))}\frac{t-\tau}{|w-y|^{d+\alpha}}f(y)\rd y\le \frac{Ct\|f\|_1}{|w-y_0|^{d+\alpha}}\,.
\end{split}
\end{align}
Let 
\begin{align*}
W_1&:=\{w\in \cV_{B}\setminus B: |x_0-w|<|y_0-w|\},\\
W_2&:=\{w\in \cV_{B}\setminus B: |x_0-w|\ge|y_0-w|\}.
\end{align*}
Using \eqref{sg_bd}, we have
\begin{align}\label{step1}
&\E^{x}\left[\1_{\{\tau\le t/2\}}P_{t-\tau}f(X_{\tau})\right]\nn\\
&=\E^{x}\left[\1_{\{\tau\le t/2\}}\1_{\{X_\tau\in W_1\}}P_{t-\tau}f(X_{\tau})\right]+\E^{x}\left[\1_{\{\tau\le t/2\}}\1_{\{X_\tau\in W_2\}}P_{t-\tau}f(X_{\tau})\right]\nn\\
&\le Ct\|f\|_1\left(\E^{x}\left[\1_{\{\tau\le t/2\}}\1_{\{X_{\tau}\in W_1\}}\frac{1}{|X_{\tau}-y_0|^{d+\alpha}}\right]+\E^{x}\left[\1_{\{\tau\le t/2\}}\1_{\{X_{\tau}\in W_2\}}\frac{1}{|X_{\tau}-y_0|^{d+\alpha}}\right]\right)\nn\\
&=:Ct\|f\|_1(I+II).
\end{align}

Now, we observe that for $w\in W_1$, 
\begin{align}\label{ineq_W1}
|x_0-y_0|\le  |x_0-w|+|y_0-w|\le 2|y_0-w|.
\end{align}
For $v\in B$ and $w\in W_2$, we also see that
\begin{align}\label{ineq_W2}
|x_0-y_0|\le 4|v-w|.
\end{align}
Indeed, $w\in W_2$ implies that $|x_0-w|\ge \frac12|x_0-y_0|=\frac{C_1}{2} R_2$. Using this and $R_2\ge R_1/C_*$, we have that for $v\in B(x_0, R_1/(8C_*))$,
$$|v-w|\ge |x_0-w|-|v-x_0|\ge \frac{R_2}{2}-\frac{R_1}{8C_*}\ge \frac{R_1}{4C_*}\ge |x_0-v|.$$
Thus, for $w\in W_2$ and $v\in B$ 
$$|x_0-y_0|\le|x_0-w|+|w-y_0| \le  2|x_0-w|\le 2(|x_0-v|+|v-w|)\le 4|v-w|.$$
By \eqref{ineq_W1} and \eqref{exit_est},
\begin{align}\label{step2}
I&=\E^{x}\left[\1_{\{\tau\le t/2\}}\1_{\{X_{\tau}\in W_1\}}\frac{t}{|X_{\tau}-y_0|^{d+\alpha}}\right]
\le \frac{C\bP^{x}(\tau\le t/2)}{|x_0-y_0|^{d+\alpha}}\le \frac{C}{|x_0-y_0|^{d+\alpha}}\frac{t}{R_1^{\alpha}}.
\end{align}
By the L\'evy system \eqref{eq:LS}, \eqref{ineq_W2} and $\cV_{B}\subset B(y_0, R_1/(8C_*))^c$,
\begin{align}\label{step3}
II&\le \E^{x}\int^{\tau\wedge t/2}_{0}\int_{W_2}\frac{1}{|w-y_0|^{d+\alpha}}\frac{1}{|X_s-w|^{d+\alpha}}\rd w \rd s\nn\\
&\le \frac{4^{d+\alpha}}{|x_0-y_0|^{d+\alpha}}\E^{x}\int^{\tau\wedge t/2}_{0}\int_{W_2}\frac{1}{|w-y_0|^{d+\alpha}} \rd w \rd s\nn\\
&\le \frac{Ct}{|x_0-y_0|^{d+\alpha}}\int_{B(y_0, R_1/(8C_*))^c}\frac{1}{|w-y_0|^{d+\alpha}} \rd w\nn\\
&\le \frac{Ct}{|x_0-y_0|^{d+\alpha}}\frac{c}{R_1^{\alpha}}.
\end{align}
Thus, by \eqref{step1}, \eqref{step2} and \eqref{step3}, we obtain the desired result.
\end{proof}

\bigskip

Since we have \autoref{p:sub_key}, we can prove \autoref{t:sub} by the same arguments in \cite{KKK19, KK21}. The proof is much simpler than the ones in \cite{KKK19, KK21} since we need not use iterative scheme. 

\medskip

\begin{proof}[Proof of \autoref{t:sub}]
Consider non-negative Borel functions $f, g$ on $\R^d$ supported in $B(y_0, \frac{\rho}{8C_*})$ and $B(x_0, \frac{\rho}{8C_*})$, respectively.  
We apply 
\autoref{lem:LBGK} with functions $f, g$,  subsets $U:=B(x_0, s), V:=B(y_0, s)$ for some $s>0$, $a=b=t/2$ and $\tau=\tau_{U}, \tau^{'}=\tau_{V}$.
The first term of the right hand side of \eqref{eq:LBGK} is
\begin{align*}
\left\langle \E^{\cdot}\left[\1_{\{\tau\le t/2\}}P_{t-\tau}f(X_{\tau})\right], g\right\rangle=\int_{B(x_0, \frac{\rho}{8C_*})}\E^{x}\left[\1_{\{\tau\le t/2\}}P_{t-\tau}f(X_{\tau})\right]\,g(x)\, \rd x ,
\end{align*}
and a similar identity holds for the second term. 
By \eqref{e:sub_key}, 
\begin{align*}
&\left\langle\E^{\cdot}\left[\1_{\{\tau\le t/2\}}P_{t-\tau}f(X_{\tau})\right], g\right\rangle\nn\\
&\le  \,C t^{-d/\alpha} \|f\|_1\|g\|_1  \left(1\wedge\frac{t}{|x_0-y_0|^{\alpha}}\right)^{1+d/\alpha}\left(1\wedge\frac{t}{(\dist(y_0-x_0,\cV))^\alpha}\right).
\end{align*}
Similarly we obtain the second term of right hand side of \eqref{eq:LBGK} and thus,
\begin{align*}
\left\langle P_tf, g\right\rangle\le Ct^{-d/\alpha} \|f\|_1\|g\|_1\left(1\wedge\frac{t}{|x_0-y_0|^{\alpha}}\right)^{1+d/\alpha}\left(1\wedge\frac{t}{(\dist(y_0-x_0,\cV))^\alpha}\right).
\end{align*} 
Since $P_tf(x)=\int_{\Rd} p(t, x, y)f(y)\rd y$ and $p$ is a continuous function, we obtain that for $t>0$ and $x_0, y_0\in\R^d$,
 \begin{align*}
 p(t, x_0,y_0)
& \le C t^{-d/\alpha}\left(1\wedge\frac{t}{|x_0-y_0|^{\alpha}}\right)^{1+d/\alpha}\left(1\wedge\frac{t}{(\dist(y_0-x_0,\cV))^\alpha}\right).
 \end{align*}
This proves  \autoref{t:sub}.
\end{proof}

\vspace{5mm}

\section{Lower bound estimates}\label{s:lbe}
In this section, we give sharp lower bound of the heat kernel. First, we show near diagonal lower bound by using survival time estimate and H\"older continuity of the heat kernel.

\begin{proposition}\label{p:ndl}
There exist positive constants $c$ and $\eps$ such that
\begin{align*}
p(t,x,y)\ge c t^{-d/\alpha}\qquad\text{for}\;\;\;y\in B(x, \eps t^{1/\alpha}).
\end{align*}
\end{proposition}
\begin{proof}
By \autoref{t:uhk}(ii), we see that there exists $c_1>0$ such that
$$\bP^x(\tau_{B(x,r)}<t)\le c_1 tr^{-\alpha}.$$
Thus,
\begin{align*}
\int_{\R^{d}\setminus B(x, (2c_1 t)^{1/\alpha})}p(t/2, x,y)\rd y \le \bP^x(\tau_{B(x, (2c_1 t)^{1/\alpha})}<t)\le \frac12.
\end{align*}
Thus, by Jensen's inequality,
\begin{align}\label{odl}
p(t,x,x)&=\int_{\R^{d}}p(t/2, x,y)^2\rd y
\ge \int_{B(x, (2c_1 t)^{1/\alpha})}p(t/2, x,y)^2\rd y\nn\\
&\ge \frac{1}{|B(x, (2c_1 t)^{1/\alpha})|}\bigg(\int_{B(x, (2c_1 t)^{1/\alpha})}p(t/2, x,y)\rd y\bigg)^2
\ge c_2 t^{-d/\alpha}.
\end{align}
Note that $c_2$ is  independent of $t>0$ and $x\in \R^{d}$. 

On the other hand, by the H\"older continuity for $p(t,x,\cdot)$ in \cite{FK13}, we can take $\eps=\eps(c_2)$ such that
\begin{align*}
|p(t,x,y)-p(t,x,z)|\le \frac{c_2}{2} t^{-d/\alpha},\quad\text{for all}\;\;y,z\in B(x, \eps t^{1/\alpha}).
\end{align*}
Thus, by \eqref{odl} and the above inequality for $y\in B(x, \eps t^{1/\alpha})$,
\begin{align*}
p(t,x,y)\ge p(t,x,x)-\frac{c_2}{2}t^{-d/\alpha}\ge \frac{c_2}{2} t^{-d/\alpha}.
\end{align*}
\end{proof}

\medskip

The following proposition follows from \autoref{p:ndl} and \cite[Proposition 3.1]{CKW21}.
\begin{proposition}\label{conservative}
The process $X$ is conservative; that is, $X$ has infinite lifetime.
\end{proposition}
Using the above results, we give estimates of mean exit time of a ball.
\begin{proposition}\label{p:meanexit}
(i) There exists a constant $c_1>0$ such that for $x_0\in \R^{d}$ and $r>0$,
\begin{align*}
\bE^x[\tau_{B(x_0,r)}]\le c_1 r^{\alpha}
\end{align*}
for all $x\in B(x_0,r)$.\\
(ii) There exists a constant $c_2>0$ such that for $x\in\Rd$ and $r>0$,
\begin{align*}
\bE^x[\tau_{B(x,r)}]\ge c_2 r^{\alpha}.
\end{align*}
\end{proposition}
\begin{proof}
(i) Let $C>0$ be the constant in  \autoref{p:uhkd}(ii). Take large $c_3$ so that $2\pi^{d/2}C \le \Gamma(d/2+1)c_3^{d/\alpha}$, where $r\mapsto\Gamma(r)$ is the gamma function. Then, for every $r>0$, $x_0\in\R^{d}$ and $x\in B(x_0,r)$, with $t:=c_3 r^{\alpha}$, we have by  \autoref{p:uhkd}(ii)
\begin{align*}
\bP^x(X_t\in B(x_0, r))=\int_{B(x_0, r)} p(t,x,y)\rd y \le \frac{C |B(x_0, r)|}{t^{d/\alpha}}= \frac{C\pi^{d/2}r^d}{\Gamma(d/2+1)c_3^{d/\alpha} r^d}\le \frac12.
\end{align*}  
Since $X$ is  conservative, it follows that for every $x\in B(x_0, r)$,
\begin{align*}
\bP^x(\tau_{B(x_0, r)}\le t)\ge\bP^x(X_t\notin B(x_0, r))\ge\frac12,
\end{align*}  
which implies $\bP^x(\tau_{B(x_0, r)}> t)\le 1/2.$ By the strong Markov property, for integer $k\ge1$,
\begin{align*}
\bP^x\big(\tau_{B(x_0, r)}> (k+1)t\big)\le\bE^x\Big[\bP^{X_{kt}}(\tau_{B(x_0, r)}> t);\tau_{B(x_0, r)}>kt\Big]\le\frac12\bP^{x}(\tau_{B(x_0, r)}> kt).
\end{align*}  
Using induction, we obtain that for every $k\ge1$,
\begin{align*}
\bP^x(\tau_{B(x_0, r)}> kt)\le 2^{-k},
\end{align*}  
which implies that
\begin{align*}
\bE^x[\tau_{B(x_0, r)}]\le \sum^{\infty}_{k=0}t(k+1)\bP^x(\tau_{B(x_0, r)}> kt)\le c_4 r^{\alpha}.
\end{align*}  

\noindent(ii) Let $c>0$ be the constant in \eqref{exit_est} and  $t:= r^{\alpha}/(2c)$. Then,
\begin{align*}
\bE^x[\tau_{B(x_0, r)}]\ge t\;\bP^x(\tau_{B(x_0, r)}\ge t)
=t\big(1-\bP^x(\tau_{B(x_0, r)}< t)\big)\ge t(1-1/2)=\frac{r^{\alpha}}{4c}.
\end{align*}
\end{proof}

Now, we will find the lower bound of $p(t,x,y)$ for all $t>0$ and $x,y\in\R^d$. If $y-x\in \cV$, then we apply the method used for isotropic processes. If $y-x\notin \cV$, then we use semigroup property and the result for the case $y-x\in \cV$. 
\begin{theorem}\label{t:lhk}
There exist a constant $C>0$ such that for any $(t,x,y)\in (0,\infty)\times \R^d\times \R^d$,
\begin{align}\label{lhk}
p(t,x,y)\ge C\,t^{-d/\alpha}\left(1 \wedge\frac{t}{|x-y|^{\alpha}}\right)^{1+d/\alpha}\left(1\wedge\frac{t}{(\dist(y-x,\cV))^{\alpha}}\right).
\end{align}
\end{theorem}
\begin{proof}
Let $\eps>0$ be the constant in \autoref{p:ndl}. Fix $x,y\in\R^d$ and $t>0$. If $|x-y|<\eps t^{1/\alpha}$, then \eqref{lhk} holds by \autoref{p:ndl} and \eqref{dist_comp}. Thus, for the rest of proof, we assume that $|x-y|\ge\eps t^{1/\alpha}$.

\smallskip

({\it Case 1}) $\1_{\cV}(x-y)=1$.

In this case, $\dist(y-x,\cV)=0$ and thus $\left(1\wedge\frac{t}{(\dist(y-x,\cV))^{\alpha}}\right)=1$. 
By \eqref{exit_est}, there exists $c_1>0$ such that 
\begin{align}\label{e:nbe1}
\bP^x(\tau_{B(x,r)} \le t) \le \frac{c_1 t}{r^{\alpha}}
\end{align}
for any $t,r>0$ and $x \in \R^d$. 
Let $\eps_1=(1/2)^{1/\alpha}\eps\in (0, \eps)$ so that for all $b \in(0, 1/2]$
\begin{align}\label{e:nbe3}
\eps (1-b)^{1/\alpha}t^{1/\alpha} \ge \eps_1 t^{1/\alpha} ,  \quad \text{for all } t>0.
\end{align}
Choose small $b\in(0,1/2)$ satisfying $b<\frac{1}{2c_1}(\eps_1/6)^{\alpha}$. Then, we see by \eqref{e:nbe1} that
\begin{align}
\label{e:nbe4}
 \bP^x(\tau_{B(x,2\eps_1 t^{1/\alpha}/3)} \le b t)\le  \bP^x(\tau_{B(x,\eps_1 t^{1/\alpha}/6)} \le b t) \le 1/2, \quad \text{for all } t>0 \text{ and } x \in \R^d.
 \end{align}
Using \autoref{p:ndl}, symmetry of $p(t,\cdot,\cdot)$ and \eqref{e:nbe3}, 
\begin{align*}
p(t,x,y) &\ge \int_{B(y,\eps (1-b)^{1/\alpha}t^{1/\alpha})} p(b t,x,z) p((1-b)t,z,y) \rd z \\
&\ge \inf_{z \in B(y,\eps (1-b)^{1/\alpha}t^{1/\alpha})} p((1-b)t,z,y) \int_{B(y,\eps (1-b)^{1/\alpha}t^{1/\alpha})} p(b t,x,z) \rd z \\
&\ge c_0t^{-d/\alpha} \bP^x ( X_{b t} \in B(y,\eps_1 t^{1/\alpha})).
\end{align*}
Thus, for \eqref{lhk}, it suffices to prove that
\begin{align}
\label{e:nbe2}
\bP^x ( X_{b t} \in B(y,\eps_1 t^{1/\alpha}))  \ge c_2\frac{t^{1+d/\alpha}}{|x-y|^{d+\alpha}}.
\end{align}
For $A \subset \R^d$, let $\sigma_A := \inf\{t>0:X_t \in A \}$.
Using  \eqref{e:nbe4} and the strong Markov property we have 
\begin{align*}
& \bP^x(X_{b t} \in B(y,\eps_1 t^{1/\alpha})) \\
&\ge \bP^x \left( \sigma_{B(y,\eps_1t^{1/\alpha}/3)} \le b t; \sup_{s \in [\sigma_{B(y,\eps_1t^{1/\alpha}/3)},b t]} \big|X_s - X_{\sigma_{B(y,\eps_1t^{1/\alpha}/3)}} \big| \le \frac{2\eps_1t^{1/\alpha}}{3} \right) \\
&\ge \bP^x (\sigma_{B(y,\eps_1t^{1/\alpha}/3)} \le b t ) \inf_{z \in B(y,\eps_1 t^{1/\alpha}/3)} \bP^z (\tau_{B(z,2\eps_1 t^{1/\alpha}/3)} >b t) \\
&\ge \frac{1}{2}\bP^x (\sigma_{B(y,\eps_1t^{1/\alpha}/3)} \le b t) \\
&\ge \frac{1}{2}\bP^x \Big(X_{(b t) \land \tau_{B(x,2\eps_1 t^{1/\alpha}/3)}} \in B(y,\eps_1t^{1/\alpha}/3) \Big).
\end{align*} 
Since $|x-y| \ge \eps t^{1/\alpha} > \eps_1 t^{1/\alpha}$, it is easy to see that $B(y, \eps_1 t^{1/\alpha}/3) \subset \Rd\setminus\overline{B(x,2\eps_1t^{1/\alpha}/3)}$. Thus, by the L\'evy system and \eqref{J_comp}, we have 
\begin{align}\label{e:nbe5}
&\bP^x \Big(X_{(b t) \land \tau_{B(x,2\eps_1 t^{1/\alpha}/3)}} \in B(y,\eps_1t^{1/\alpha}/3) \Big)\nn \\
&  = \E^x \left[ \sum_{s \le (b t) \land \tau_{B(x,2\eps_1 t^{1/\alpha}/3)}} \1_{ \{ X_s \in B(y,\eps_1t^{1/\alpha}/3)  \}  }   \right]\nn\\
&  \ge \E^x \Bigg[ \int_0^{(b t) \land \tau_{B(x,2\eps_1 t^{1/\alpha}/3)}}  \int_{B(y, \eps_1 t^{1/\alpha}/3)} J(X_s,u) \,\rd u \,  \rd s\Bigg] \nn\\
&  \ge c_3\E^x \Bigg[ \int_0^{(b t) \land \tau_{B(x,\eps_1 t^{1/\alpha}/6)}}  \int_{B(y, \eps_1 t^{1/\alpha}/3)} \frac{\1_{\cV}(X_s-u)}{|X_s-u|^{d+\alpha}} \,\rd u\,  \rd s \Bigg]\nn\\
&  = c_3\E^x \Bigg[ \int_0^{(b t) \land \tau_{B(x,\eps_1 t^{1/\alpha}/6)}}  \int_{\cV_{X_s}\cap B(y, \eps_1 t^{1/\alpha}/3)} \frac{1}{|X_s-u|^{d+\alpha}} \,\rd u\,  \rd s \Bigg].
\end{align}
Since $y\in \cV_x$, we have 
$$\sup_{v\in B(x,\eps_1 t^{1/\alpha}/6)}\text{dist}(y, \cV_v)\le {\eps_1t^{1/\alpha}}/{6}.$$
Thus, for $v\in B(x,\eps_1 t^{1/\alpha}/6)$,
$$|\cV_v \cap B(y, \eps_1 t^{1/\alpha}/3)|\ge c_4 t^{d/\alpha},$$
where $c_4=c_4(\eps_1, \theta)>0$ is independent of $t, x, y$. 
Moreover, for $v\in B(x,\eps_1 t^{1/\alpha}/6)$ and $u\in B(y, \eps_1 t^{1/\alpha}/3)$,
$$|v -u | \le |v -x| + |x-y| + |y-u| \le |x-y| + \eps_1 t^{1/\alpha} \le 2|x-y|. $$
Using these observation, \eqref{e:nbe5} and \eqref{e:nbe4},
\begin{align*}
&\bP^x \Big(X_{(b t) \land \tau_{B(x, 2\eps_1 t^{1/\alpha}/3)}} \in B(y,\eps_1t^{1/\alpha}/3) \Big) 
\ge c_5\E^x \Big[(b t) \land \tau_{B(x, \eps_1 t^{1/\alpha}/6)} \Big]  \frac{c_4t^{d/\alpha}}{|x-y|^{d+\alpha}} \\
& \ge c_4c_5 (b t) \bP^x\big( \tau_{B(x, \eps_1 t^{1/\alpha}/6)} \ge b t \big) \frac{t^{d/\alpha}}{|x-y|^{d+\alpha}} 
\ge c_4c_5b 2^{-1}  \frac{t^{1+d/\alpha}}{|x-y|^{d+\alpha}},
\end{align*}
which proves \eqref{e:nbe2} for $x,y\in\R^d$ with $y\in \cV_x$.

\medskip

({\it Case 2}) $\1_{\cV}(x-y)=0$.

\smallskip

Recall that $\cV$ is defined by \eqref{general_set}. Let $\Gamma_i\subset \cV$ be a symmetric cone such that $\dist(y-x,\cV)=\dist(y-x, \Gamma_i)$. For simplicity, we will write $\Gamma=\Gamma_i$. 
Let $z_0\in S(\Gamma, x,y)\subset\partial \Gamma_x\cap \partial \Gamma_y$. Then, $x,y\in \partial \Gamma_{z_0}$. Let $R_1:=|x-z_0|$ and $R_2:=|y-z_0|$.
Without loss of generality, we assume that $R_1\ge R_2$ and $z_0\in \partial\Gamma^+_{x}$. Then, $x, y\in \partial\Gamma^-_{z_0}$, $\Gamma^+_{z_0}\subset \Gamma^+_x\cap\Gamma^+_y$ and $R_2\asymp \dist(y-x,\Gamma)=\dist(y-x,\cV)$.  
We consider two cases $R_2>\frac{\eps}{2}(t/2)^{1/\alpha}$ and $R_2\le \frac{\eps}{2}(t/2)^{1/\alpha}$ separately.

\medskip

({\it Case 2-i}) $R_2\le \frac{\eps}{2}(t/2)^{1/\alpha}$.

\smallskip

Since we consider $x,y$ with $|x-y|\ge\eps t^{1/\alpha}$, we see that
 $R_1\ge\frac12|x-y|\ge\frac{\eps}{2}t^{1/\alpha}$ by \eqref{dist_comp}. Let $z\in B(z_0, \frac{\eps}{2}(t/2)^{1/\alpha})\cap \Gamma^+_{z_0}$. Then,
\begin{align*}
(1-2^{-1/\alpha})R_1\le |x-z_0|-|z_0-z|\le |x-z|\le |x-z_0|+|z_0-z|\le 2R_1.
\end{align*}
Moreover, $|y-z|\le |y-z_0|+|z_0-z|<\eps (t/2)^{1/\alpha}.$  
 Thus, by the result of ({\it Case 1}), \autoref{p:ndl} and $R_1^{\alpha}\ge({\eps}/{2})^{\alpha}t$,
\begin{align}\label{l_c22}
p(t,x,y) &\ge \int_{B(z_0, \frac{\eps}{2}(t/2)^{1/\alpha})\cap \Gamma_{z_0}} p(t/2,x,z) p(t/2,z,y) \rd z \nn\\
&\ge C\int_{B(z_0, \frac{\eps}{2}(t/2)^{1/\alpha})\cap \Gamma^+_{z_0}} t^{-d/\alpha}\Big(1\wedge \frac{t}{|x-z|^{\alpha}}\Big)^{1+d/\alpha}t^{-d/\alpha} \rd z \nn\\
&\ge Ct^{-d/\alpha}\Big(\frac{t}{R_1^{\alpha}}\Big)^{1+d/\alpha}\int_{B(z_0, \frac{\eps}{2}(t/2)^{1/\alpha})\cap \Gamma^+_{z_0}} t^{-d/\alpha} \,\rd z \nn\\
&\ge Ct^{-d/\alpha}\Big(\frac{t}{R_1^{\alpha}}\Big)^{1+d/\alpha}\nn\\
&\ge Ct^{-d/\alpha}\Big(\frac{t}{R_1^{\alpha}}\Big)^{1+d/\alpha}\Big(1\wedge \frac{t}{R_2^{\alpha}}\Big). 
\end{align}

\smallskip

({\it Case 2-ii}) $R_2>\frac{\eps}{2}(t/2)^{1/\alpha}$.

\smallskip
Let $\lambda=\lambda_i$ and $\theta=\theta_i$, where $\lambda_i$ and $\theta_i$ are the axis and  aperture of $\Gamma$, respectively.   We first observe that  
\begin{align*}
\text{dist}(x, \Gamma^+_{z_0})=
\begin{cases}
\;|x-z_0|&\mbox{if}\;\;\theta\in(0,\pi/4);\\
\;|x-z_0|\sin(\pi-2\theta)&\mbox{if}\;\;\theta\in[\pi/4,\pi/2).
\end{cases}
\end{align*}
This shows that $|x-z_0|\sin(\pi-2\theta)\le |x-z|$ for $z\in \Gamma^+_{z_0}$.
Moreover, it is easy to see that  for $z\in B(z_0, 5R_1)$, 
$$|x-z|\le |x-z_0|+|z_0-z|\le 6|x-z_0|.$$
Thus, $|x-z|\asymp R_1$ for all $z\in \Gamma^+_{z_0}\cap B(z_0, 5R_1)$.

Let $y_1=y+2R_2(\cos\theta)\lambda$.  Then, we see that $y_1\in\partial \Gamma^+_{z_0}$ and $\overline{yy_1}$ is parallel with $\lambda$. 
Let $y_2=2(y_1-z_0)-z_0$ and $y_3=3(y_1-z_0)-z_0$. Then, $y_2, y_3\in\partial\Gamma^+_{z_0}$ and 
\begin{align*}
|\overline{yy_2}|&=R_2\sqrt{9\cos^2\theta+\sin^2\theta}=:c_6R_2,\\ 
|\overline{yy_3}|&=R_2\sqrt{16\cos^2\theta+4\sin^2\theta}=:c_7R_2.
\end{align*}
Take $\theta'\in(0,\theta)$ such that $\tan\theta'=\frac{\sin\theta}{3\cos\theta}$. Then, we see that $y_2\in\partial\Gamma^+_y(\lambda, \theta')$ and 
$$\Gamma^+_y(\lambda, \theta')\setminus B(y, c_6R_2)\subset \Gamma^+_{z_0}.$$
Using that $R_1\ge R_2$, we have $B(y, c_7R_2)\subset B(z_0, 5R_1)$. Indeed, for $w\in B(y, c_7R_2)$, 
$$|z_0-w|\le |z_0-y|+|y-w|\le R_2+c_7R_2< 5R_2\le 5R_1.$$
Combining the above observations, we obtain
\begin{align}\label{l_case2i}
\Gamma^+_y(\lambda, \theta')\cap (B(y, c_7R_2)\setminus B(y, c_6R_2))\subset \Gamma^+_{z_0}\cap B(z_0, 5R_1),
\end{align}
Since $R_2>\frac{\eps}{2}(t/2)^{1/\alpha}$, for $z\in B(y, c_7R_2)\setminus B(y, c_6R_2)$, we have that $|y-z|\asymp R_2>\frac{\eps}{2}(t/2)^{1/\alpha}$. This gives that there exists $c>0$ such that for $z\in  B(y, c_7R_2)\setminus B(y, c_6R_2)$,
$$\Big(t^{-d/\alpha}\wedge \frac{t}{|y-z|^{d+\alpha}}\Big)\ge \frac{ct}{|y-z|^{d+\alpha}}.$$
Using this, \eqref{l_case2i} and the result of ({\it Case 1}), 
\begin{align}\label{l_c21}
p(t,x,y) &\ge \int_{\Gamma^+_{z_0}\cap B(z_0, 5R_1)} p(t/2,x,z) p(t/2,z,y) \rd z \nn\\
&\ge C\int_{\Gamma^+_{z_0}\cap B(z_0, 5R_1)} t^{-d/\alpha}\Big(1\wedge \frac{t}{|x-z|^{\alpha}}\Big)^{1+d/\alpha}\Big(t^{-d/\alpha}\wedge \frac{t}{|y-z|^{d+\alpha}}\Big) \rd z \nn\\
&\ge Ct^{-d/\alpha}\Big(\frac{t}{R_1^{\alpha}}\Big)^{1+d/\alpha}\int_{\Gamma^+_y(\lambda, \theta')\cap (B(y, c_7R_2)\setminus B(y, c_6R_2))}\Big(t^{-d/\alpha}\wedge \frac{t}{|y-z|^{d+\alpha}}\Big) \rd z \nn\\
&\ge Ct^{-d/\alpha}\Big(\frac{t}{R_1^{\alpha}}\Big)^{1+d/\alpha}\int_{\Gamma^+_y(\lambda, \theta')\cap (B(y, c_7R_2)\setminus B(y, c_6R_2))}\frac{t}{|y-z|^{d+\alpha}} \rd z \nn\\
&\ge Ct^{-d/\alpha}\Big(\frac{t}{R_1^{\alpha}}\Big)^{1+d/\alpha}\int_{c_6R_2}^{c_7R_2}\frac{t}{s^{1+\alpha}} \rd s \nn\\
&\ge Ct^{-d/\alpha}\Big(\frac{t}{R_1^{\alpha}}\Big)^{1+d/\alpha}\Big(\frac{t}{R_2^{\alpha}}\Big), 
\end{align}
where $C=C(\theta')>0$.

Thus, combining \eqref{l_c22} and \eqref{l_c21}, we obtain the desired lower bound for ({\it Case 2}).

\end{proof}

\medskip

\begin{proof}[Proof of \autoref{t:main}]
The existence and H\"older continuity of $p(t,x,y)$ follow from \autoref{p:uhkd}(ii) and \autoref{r:Holder}. The H\"older continuity of heat kernel and \autoref{conservative} show that the Hunt process $X$ associated with $(\cE, \cF)$ is conservative Feller process. The heat kernel bounds \eqref{e:main} follow from  \autoref{t:sub} and \autoref{t:lhk}.

\end{proof}

\medskip

\begin{proof}[Proof of \autoref{c:GFE}]
Recall that $d\ge2>\alpha$. By \autoref{t:main}, we see that there exists a constant $C>0$ such that 
$$p(t,x,y)\le C\left(t^{-d/\alpha}\wedge \frac{t}{|x-y|^{d+\alpha}}\right).$$ 
Thus, we obtain
\begin{align*}
G(x,y)\le C\int^{|x-y|^{\alpha}}_{0}\frac{t}{|x-y|^{d+\alpha}}\rd t+C\int_{|x-y|^{\alpha}}^{\infty}t^{-d/\alpha}\rd t\le \frac{C'}{|x-y|^{d-\alpha}}.
\end{align*}
For the lower bound, we observe that $|x-y|\ge \dist(y-x, \cV)$ for any $x,y\in\R^d$. Thus, by \autoref{t:main}, there exists a constant $c>0$ such that
\begin{align*}
p(t,x,y)\ge ct^{-d/\alpha}\quad\text{for}\;\;t>|x-y|^\alpha.
\end{align*}
Thus, we have that for any $x,y\in\R^d$
\begin{align*}
G(x,y)\ge \int^{\infty}_{|x-y|^{\alpha}}p(t,x,y)\rd t\ge c\int^{\infty}_{|x-y|^{\alpha}}t^{-d/\alpha}\rd t=\frac{c'}{|x-y|^{d-\alpha}}.
\end{align*}

\end{proof}

\medskip

{\bf Acknowledgements.}  The author is grateful to Marvin Weidner for giving helpful comments on the  regularity of solutions to equations with nonlocal operators.

\begin{singlespace}

\end{singlespace}

\vskip 0.1truein

\parindent=0em

\end{document}